\documentclass{amsart}
\usepackage{amsmath}
\usepackage{amssymb}
\usepackage{amsfonts}

\newtheorem{theorem}{Theorem}
\theoremstyle{plain}

\newtheorem{definition}{Definition}
\newtheorem{example}{Example}

\newtheorem{lemma}{Lemma}

\newtheorem{remark}{Remark}

\numberwithin{equation}{section}
\input{tcilatex}

\begin{document}
\title[]{ON THE $F$-CONTRACTION PROPERTIES OF MULTIVALUED INTEGRAL TYPE
TRANSFORMATIONS}
\author{Derya Sekman}
\address{Department of Mathematics, Faculty of Arts and Sciences, Ahi Evran
University, 40100 K\i r\c{s}ehir, Turkey}
\email{derya.sekman@ahievran.edu.tr}
\urladdr{}
\thanks{}
\author{Vatan Karakaya}
\address{Department of Mathematical Engineering, Faculty of
Chemistry-Metallurgical, Y\i ld\i z Technical University, 34210 Istanbul,
Turkey}
\email{vkkaya@yildiz.edu.tr}
\urladdr{}
\dedicatory{}
\thanks{}

\begin{abstract}
The main purpose of this work is to extend the properties of multivalued
transformations to the integral type transformations and to obtain the
existence of fixed points under $F$-contraction. In addition, the results of
this study were evaluated with some interesting example.
\end{abstract}

\maketitle

\section{\protect\bigskip Introduction and Preliminaries}

In the development process of fixed point theory, every new problem has
contributed to the expansion of this theory and it has been possible to
incorporate new fields and new concepts into it. In this sense, the
existence of fixed points of multivalued transformations is an important
research area. These studies, which started with Nadler \cite{nad}, leaded
many new ideas. In this way, many authors including Gordji \cite{gord},
Berinde \cite{ber}, Sekman et al. \cite{sekman} the others studied and
generalized set-valued mappings. The fact that the transformation classes
have the contraction condition, first given by Banach \cite{ban}, has an
important role in the existence and uniqueness of the fixed point. Quiet
recently, Wardowski \cite{ward} introduced the concept of $F$-contraction
and examined the properties of this transformation for single valued
transformations. Later, many authors have done a lot of work using this
idea. One can find in list \cite{acar,ekp,minak,sgroi}. Afterward, by
combining the concepts of $F$-contraction and multivalued mapping, Altun et
al. \cite{altun} defined the multivalued $F$-contraction. Our goal in this
work is to determine the behavior of multivalued integral mappings under $F$%
-contraction and to investigate their some properties, which is a open
problem in the current literature. It also supports the results of the study
with some interesting examples.

Let $(X,d)$ be a metric space. Denote by $P(X)$ the family of all nonempty
subsets of $X$, $CB(X)$ the family of all nonempty, closed and bounded
subsets of $X$ and $K(X)$ the family of all nonempty compact subsets of $X$.

\begin{definition}
Let $X$ be a nonempty set. $T$ is said to be a multivalued mapping if $T$ is
a mapping from $X$ to the power set of $X.$ We denote a multivalued map by $%
T:X\rightarrow P(X).$
\end{definition}

\begin{definition}[see;\protect\cite{kelley}]
Let $(X,d)$ be a complete metric space. We define the Hausdorff metric on $%
CB(X)$ by%
\begin{equation*}
{\small H(A,B):=}\max \left\{ \underset{x\in A}{\sup }{\small D(x,B),}%
\underset{y\in B}{\sup }{\small D(y,A)}\right\} {\small ,}
\end{equation*}%
for all $A,B\in CB(X),$ where $D(x,B):=\underset{b\in B}{\inf }d(x,b)$ for
all $x\in X.$ Mapping $H$ is said to be a \textit{Hausdorff metric} induced
by $d.$
\end{definition}

\begin{lemma}[see;\protect\cite{dube}]
\label{dube}Let $A,B\in CB(X),$ then for any $a\in A,$%
\begin{equation*}
{\small D(a,B)\leq H(A,B).}
\end{equation*}
\end{lemma}

\begin{definition}[see;\protect\cite{nad}]
Let $(X,d)$ be a metric space. A map $T:X\rightarrow CB(X)$ is said to be
multivalued contraction if there exists $0\leq \lambda <1$ such that 
\begin{equation*}
{\small H(Tx,Ty)\leq \lambda d(x,y)}
\end{equation*}%
for all $x,y\in X.$
\end{definition}

\begin{definition}[see;\protect\cite{nad}]
A point $x_{0}\in X$ is said to be a fixed point of a multivalued mapping $%
T:X\rightarrow CB(X)$ such that $x_{0}\in T(x_{0}).$
\end{definition}

\begin{theorem}[see;\protect\cite{ojha}]
\label{ojha}Let $(X,d)$ be a complete metric space. Suppose $T:X\rightarrow
CB(X)$ is a contraction mapping for some $0\leq \alpha <1,$%
\begin{equation*}
\dint_{0}^{H(Tx,Ty)}{\small \varphi (t)dt\leq \alpha }\dint_{0}^{M(x,y)}%
{\small \varphi (t)dt}
\end{equation*}%
where%
\begin{equation*}
{\small M(x,y)=}\max \left\{ {\small d(x,y),D(x,Tx),D(y,Ty),}\frac{1}{2}%
\left[ {\small D(x,Ty)+D(y,Tx)}\right] \right\}
\end{equation*}%
for all $x,y\in X.$ Then there exists a point $x\in X$ such that $x\in Tx$
(i.e., $x$ is a fixed point of $T$).
\end{theorem}

Now, we will give a new type of contraction called $F$-contraction by
introduced Wardowski \cite{ward} and a fixed point theorem concerning $F$%
-contraction.

\begin{definition}[see;\protect\cite{ward}]
Let $F:$\bigskip $%
\mathbb{R}
^{+}\rightarrow 
\mathbb{R}
$ be a mapping satisfying:
\end{definition}

\begin{itemize}
\item[$(F1)$] $F$ is strictly increasing, i.e. for all $\alpha ,$ $\beta \in 
$ \bigskip $%
\mathbb{R}
^{+}$ such that $\alpha <\beta ,F(\alpha )<F(\beta ),$

\item[$(F2)$] For each sequence $\left\{ \alpha _{n}\right\} _{n\in 
\mathbb{N}
}$ of positive numbers $\underset{n\rightarrow \infty }{\lim }\alpha _{n}=0$
if and only if $\underset{n\rightarrow \infty }{\lim }F(\alpha _{n})=-\infty
,$

\item[$(F3)$] There exists $k\in (0,1)$ such that$\underset{\alpha
\rightarrow 0^{+}}{\lim }\alpha ^{k}F(\alpha )=0.$
\end{itemize}

A mapping $T:X\rightarrow X$ is said to be an $F$-contraction if there
exists $\tau >0$ such that%
\begin{equation}
{\small d(Tx,Ty)>0\Rightarrow \tau +F(d(Tx,Ty))\leq F(d(x,y))}  \label{cont}
\end{equation}%
for all $x,y\in X.$

\begin{remark}
From ($F1$) and (\ref{cont}) it is easy to conclude that every $F$%
-contraction T is a contractive mapping, i.e.%
\begin{equation*}
{\small d(Tx,Ty)<d(x,y),}\text{{\small \ for all} }{\small x,y\in X,}\text{ }%
{\small Tx\neq Ty.}
\end{equation*}
\end{remark}

Thus every $F$-contraction is a continuous mapping.

\begin{remark}
Let $F_{1},F_{2}$ be the mappings satisfying ($F1$-$F3$). If $F_{1}(\alpha
)\leq F_{2}(\alpha )$ for all $\alpha >0$ and a mapping $G:F_{1}$-$F_{2}$ is
nondecreasing then every $F_{1}$-contraction $T$ is $F_{2}$-contraction.
\end{remark}

\begin{theorem}[see;\protect\cite{ward}]
Let $(X,d)$ be a complete metric space and let $T:X\rightarrow X$ be an $F$%
-contraction. Then $T$ has a unique fixed point $x^{\ast }\in X$ and for
every $x_{0}\in X$ a sequence $\left\{ T^{n}x_{0}\right\} _{n\in 
\mathbb{N}
}$ is convergent to $x^{\ast }.$
\end{theorem}

\begin{definition}[see;\protect\cite{acar}]
Let $(X,d)$ be a metric space and $T:X\rightarrow CB(X)$ be a mapping. Then $%
T$ is said to be a generalized multivalued $F$-contraction if there exists $%
\tau >0$ such that%
\begin{equation*}
{\small H(Tx,Ty)>0\Rightarrow \tau +F(H(Tx,Ty))\leq F(M(x,y)),}
\end{equation*}%
where 
\begin{equation*}
{\small M(x,y)=}\max \left\{ {\small d(x,y),D(x,Tx),D(y,Ty),}\frac{1}{2}%
\left[ {\small D(x,Ty)+D(y,Tx)}\right] \right\} ,
\end{equation*}%
for all $x,y\in X.$
\end{definition}

\begin{theorem}[see;\protect\cite{acar}]
Let $(X,d)$ be a complete metric space and $T:X\rightarrow K(X)$ be a
generalized multivalued $F$-contraction. If $T$ or $F$ is continuous, then $%
T $ has a fixed point in $X.$
\end{theorem}

\begin{theorem}[see;\protect\cite{altun}]
Let $(X,d)$ be a complete metric space and $T:X\rightarrow K(X)$ be a
multivalued $F$-contraction, then $T$ has a fixed point in $X$.
\end{theorem}

\begin{theorem}[see;\protect\cite{altun}]
Let $(X,d)$ be a complete metric space and let $T:X\rightarrow CB(X)$ be a $%
F $-contraction. Suppose that $F$ also satisfies
\end{theorem}

\begin{itemize}
\item[(F4)] $F(\inf $ $A)=\inf $ $F(A)$ for all $A\subset (0,\infty )$ with $%
\inf $ $A>0$. Then $T$ has a fixed point.
\end{itemize}

\section{Main Result}

\begin{definition}
Let $(X,d)$ be complete metric space and $T:X\rightarrow K(X)$ be a $F$%
-contraction of generalized multivalued integral type mapping if there
exists $\tau >0$ such that, for all $x,y\in X$, 
\begin{equation}
{\small H(Tx,Ty)>0\Rightarrow \tau +F}\left( \dint_{0}^{H(Tx,Ty)}{\small %
\varphi (t)dt}\right) {\small \leq F}\left( \dint_{0}^{M(x,y)}{\small %
\varphi (t)dt}\right) ,  \label{bir}
\end{equation}%
where 
\begin{equation*}
{\small M(x,y)=}\max \left\{ {\small d(x,y),D(x,Tx),D(y,Ty),}\frac{1}{2}%
\left[ {\small D(x,Ty)+D(y,Tx)}\right] \right\} ,
\end{equation*}%
where $\varphi :[0,+\infty )\rightarrow \lbrack 0,+\infty )$ is a
Lebesque-integrable mapping which is summable on each compact subset of $%
[0,+\infty ),$ non-negative, and such that for each $\varepsilon >0,$ $%
\int_{0}^{\varepsilon }\varphi (t)dt>0.$

\begin{example}
Let $F:$\bigskip $%
\mathbb{R}
^{+}\rightarrow 
\mathbb{R}
$ be a mapping given by $F(\alpha )=\ln \alpha .$ It is clear that $F$
satisfies $(F1)$-$(F3),$ $for$ $any$ $k\in (0,1)$. Every $F$-contraction for
generalized\qquad \qquad multivalued integral type mapping satisfies%
\begin{equation*}
{\small \tau +F}\left( \dint_{0}^{H(Tx,Ty)}{\small \varphi (t)dt}\right) 
{\small \leq F}\left( \dint_{0}^{M(x,y)}{\small \varphi (t)dt}\right) 
{\small .}
\end{equation*}%
Now we have%
\begin{equation*}
\ln {\small e}^{\tau }{\small +}\ln \left( \dint_{0}^{H(Tx,Ty)}{\small %
\varphi (t)dt}\right) {\small \leq }\ln \left( \dint_{0}^{M(x,y)}{\small %
\varphi (t)dt}\right)
\end{equation*}%
\begin{equation*}
\dint_{0}^{H(Tx,Ty)}{\small \varphi (t)dt\leq e}^{-\tau }\dint_{0}^{M(x,y)}%
{\small \varphi (t)dt,}
\end{equation*}%
for all ${\small x,y\in X,}$ ${\small Tx\neq Ty}.$ It is clear that for $%
x,y\in X$ such that $Tx=Ty$ the inequality%
\begin{equation*}
\dint_{0}^{H(Tx,Ty)}{\small \varphi (t)dt\leq e}^{-\tau }\dint_{0}^{M(x,y)}%
{\small \varphi (t)dt}
\end{equation*}%
also holds, i.e. $T$ is a contraction.
\end{example}
\end{definition}

\begin{theorem}
Let $(X,d)$ be a complete metric space and $T:X\rightarrow K(X)$ be a $F$%
-contraction of generalized multivalued integral type mapping. If $T$ or $F$
is continuous, then $T$ has a fixed point in $X.$
\end{theorem}

\begin{proof}
Let $x_{0}\in X$ be an arbitrary point and define a sequence $\left\{
x_{n}\right\} $ by $x_{n+1}\in Tx_{n}$ for $n=1,2,\ldots .$ As $Tx$ $is$
nonempty for all $x\in X,$ we can choose $x_{1}\in Tx_{0}.$ If $x_{1}\in
Tx_{1},$ then $x_{1}$ is a fixed point of $T$. Let $x_{1}\notin Tx_{1},$
then $D(x_{1},Tx_{1})>0$ since $Tx_{1}$ is compact. By using $(F1)$, from
Lemma \ref{dube} and (\ref{bir}), we can write that%
\begin{eqnarray*}
{\small F}\left( \int_{0}^{D(x_{1},Tx_{1})}{\small \varphi (t)dt}\right) &%
{\small \leq }&{\small F}\left( \int_{0}^{H(Tx_{0},Tx_{1})}{\small \varphi
(t)dt}\right) {\small \leq F}\left( \int_{0}^{M(x_{0},x_{1})}{\small \varphi
(t)dt}\right) {\small -\tau } \\
&{\small =}&{\small F}\left( \int_{0}^{\max \left\{ {\small d(x}_{0}{\small %
,x}_{1}{\small ),D(x}_{0}{\small ,Tx}_{0}{\small ),D(x}_{1}{\small ,Tx}_{1}%
{\small ),}\frac{1}{2}\left[ {\small D(x}_{0}{\small ,Tx}_{1}{\small )+D(x}%
_{1}{\small ,Tx}_{0}{\small )}\right] \right\} \ \ }{\small \varphi (t)dt}%
\right) {\small -\tau } \\
&{\small \leq }&{\small F}\left( \int_{0}^{\max \left\{ {\small d(x}_{0}%
{\small ,x}_{1}{\small ),}\frac{1}{2}{\small D(x}_{0}{\small ,Tx}_{1}{\small %
)}\right\} \ \ }{\small \varphi (t)dt}\right) {\small -\tau } \\
&{\small \leq }&{\small F}\left( \int_{0}^{\max \left\{ {\small d(x}_{0}%
{\small ,x}_{1}{\small ),D(x}_{1}{\small ,Tx}_{1})\right\} \ \ }{\small %
\varphi (t)dt}\right) {\small -\tau }
\end{eqnarray*}%
\begin{equation}
{\small \leq F}\left( \int_{0}^{d(x_{0},x_{1})}{\small \varphi (t)dt}\right) 
{\small -\tau .}  \label{iki}
\end{equation}

Also, since $Tx_{1}$ is compact, we obtain that $x_{2}\in Tx_{1}$ such that $%
d(x_{1},x_{2})=$ $D(x_{1},Tx_{1}).$ From (\ref{iki}) we have%
\begin{eqnarray*}
{\small F}\left( \int_{0}^{d(x_{1},x_{2})}{\small \varphi (t)dt}\right) &%
{\small \leq }&{\small F}\left( \int_{0}^{H(Tx_{0},Tx_{1})}{\small \varphi
(t)dt}\right) {\small \leq F}\left( \int_{0}^{d(x_{0},x_{1})}{\small \varphi
(t)dt}\right) {\small -\tau } \\
{\small F}\left( \int_{0}^{d(x_{2},x_{3})}{\small \varphi (t)dt}\right) &%
{\small \leq }&{\small F}\left( \int_{0}^{d(x_{1},x_{2})}{\small \varphi
(t)dt}\right) {\small -\tau \leq F}\left( \int_{0}^{d(x_{0},x_{1})}{\small %
\varphi (t)dt}\right) {\small -2\tau } \\
&&\vdots
\end{eqnarray*}%
\begin{equation}
{\small F}\left( \int_{0}^{d(x_{n},x_{n+1})}{\small \varphi (t)dt}\right) 
{\small \leq F}\left( \int_{0}^{d(x_{n-1},x_{n})}{\small \varphi (t)dt}%
\right) {\small -\tau \leq F}\left( \int_{0}^{d(x_{0},x_{1})}{\small \varphi
(t)dt}\right) {\small -n\tau .}  \label{uc}
\end{equation}

So we obtain a sequence $\left\{ x_{n}\right\} $ in $X$ such that $%
x_{n+1}\in Tx_{n}$ and for all $n\in N.$

If there exists $n_{0}\in N$ for which $x_{n_{0}}\in Tx_{n_{0}},$ then $%
x_{n_{0}}$ is a fixed point of $T$ and so the proof is completed. Thus,
suppose that for every $n\in 
\mathbb{N}
,$ $x_{n}\notin Tx_{n}.$ Denote 
\begin{equation*}
{\small \gamma }_{n}{\small =}\int_{0}^{d(x_{n},x_{n+1})}{\small \varphi
(t)dt,}
\end{equation*}%
for $n=0,1,2,\ldots .$

Then $\gamma _{n}>0$ for all $n$ and by using (\ref{uc})%
\begin{equation}
{\small F}\left( \gamma _{n}\right) {\small \leq F}\left( \gamma
_{n-1}\right) {\small -\tau \leq F}\left( \gamma _{n-2}\right) {\small %
-2\tau \leq \cdots \leq F}\left( \gamma _{0}\right) {\small -n\tau .}
\label{dort}
\end{equation}

From (\ref{dort}), we get $\underset{n\rightarrow \infty }{lim}$ $F(\gamma
_{n})=-\infty .$ Thus, from ($F2$), we have%
\begin{equation*}
\underset{n\rightarrow \infty }{lim}{\small \gamma }_{n}{\small =0.}
\end{equation*}%
From ($F3$), there exists $k\in (0,1)$ such that 
\begin{equation*}
\underset{n\rightarrow \infty }{lim}{\small \gamma }_{n}^{k}{\small F}\left(
\gamma _{n}\right) {\small =0}.
\end{equation*}

By (\ref{dort}), the following inequality holds for all $n\in 
\mathbb{N}
$%
\begin{equation}
{\small \gamma }_{n}^{k}{\small F}\left( \gamma _{n}\right) {\small -\gamma }%
_{n}^{k}{\small F}\left( \gamma _{0}\right) {\small \leq -\gamma }_{n}^{k}%
{\small n\tau \leq 0}  \label{bes}
\end{equation}

Letting $n\rightarrow \infty $ in (\ref{bes}), we obtain that 
\begin{equation}
\underset{n\rightarrow \infty }{lim}{\small n\gamma }_{n}^{k}{\small =0}
\label{alt}
\end{equation}

From (\ref{alt}), there exists $n_{1}\in 
\mathbb{N}
$ such that $n\gamma _{n}^{k}\leq 1$ for all $n\geq n_{1}.$ So we obtain
that 
\begin{equation}
{\small \gamma }_{n}{\small \leq }\frac{1}{n^{1/k}}{\small .}  \label{yedi}
\end{equation}

Now, let $m,n\in 
\mathbb{N}
$ such that $m>n>n_{1}$ to show that $\left\{ x_{n}\right\} $ is a Cauchy
sequence. By using the triangle inequality for the metric and from (\ref%
{yedi}), we have 
\begin{eqnarray*}
\int_{0}^{d(x_{n},x_{m})}{\small \varphi (t)dt} &{\small \leq }%
&\int_{0}^{d(x_{n},x_{n+1})}{\small \varphi (t)dt+}%
\int_{0}^{d(x_{n+1},x_{n+2})}{\small \varphi (t)dt+\cdots +}%
\int_{0}^{d(x_{m-1},x_{m})}{\small \varphi (t)dt} \\
&{\small =}&{\small \gamma }_{n}{\small +\gamma }_{n+1}{\small +\cdots
+\gamma }_{m-1} \\
&{\small \leq }&\underset{i=n}{\overset{\infty }{\sum }}\frac{1}{i^{1/k}}%
{\small .}
\end{eqnarray*}

By the convergence of the series $\underset{i=n}{\overset{\infty }{\sum }}%
\frac{1}{i^{1/k}},$ we get $\int_{0}^{d(x_{n},x_{m})}\varphi
(t)dt\rightarrow 0$ as $n\rightarrow \infty .$ As a result, $\left\{
x_{n}\right\} $ is a Cauchy sequence in $(X,d).$ Since $(X,d)$ is a complete
metric space, the sequence $\left\{ x_{n}\right\} $ converges to some point $%
z\in X,$ that is, $\underset{n\rightarrow \infty }{lim}x_{n}=z.$

If $T$ is compact, then we have $Tx_{n}\rightarrow Tz$ and from Lemma \ref%
{dube} 
\begin{equation*}
{\small D(x}_{n}{\small ,Tz)\leq H(Tx}_{n-1}{\small ,Tz),}
\end{equation*}

so $D(z,Tz)=0$ and $z\in Tz.$

Now, suppose $F$ is continuous. In this case, we claim that $z\in Tz.$
Assume the contrary, that is, $z\notin Tz.$ In this case, there exists an $%
n_{0}\in 
\mathbb{N}
$ and a subsequence $\left\{ x_{n_{k}}\right\} $ of $\left\{ x_{n}\right\} $
such that $D(x_{n_{k}+1},Tz)>0$ for all $n_{k}\geq n_{0}.$ (Otherwise, there
exists $n_{1}\in 
\mathbb{N}
$ such that $x_{n}\in Tz$ for all $n\geq n_{1},$ which implies that $z\in
Tz. $ This is a contradiction$).$

Since $D(x_{n_{k}+1},Tz)>0$ for all $n_{k}\geq n_{0},$ then we have 
\begin{eqnarray*}
{\small \tau +F}\left( \int_{0}^{D(x_{n_{k}+1},Tz)}{\small \varphi (t)dt}%
\right) &{\small \leq }&{\small \tau +F}\left( \int_{0}^{H(Tx_{n_{k}},Tz)}%
{\small \varphi (t)dt}\right) \\
&{\small \leq }&{\small F}\left( \int_{0}^{M(x_{n_{k}},z)}{\small \varphi
(t)dt}\right) \\
&{\small =}&{\small F}\left( \int_{0}^{\max \left\{ {\small d(x}_{n_{k}}%
{\small ,z),D(x}_{n_{k}}{\small ,Tx}_{n_{k}}{\small ),D(z,Tz),}\frac{1}{2}%
\left[ {\small D(x}_{n_{k}}{\small ,Tz)+D(z,Tx}_{n_{k}}{\small )}\right]
\right\} \ \ }{\small \varphi (t)dt}\right) .
\end{eqnarray*}

Taking the limit as $k\rightarrow \infty $ and using the continuity of $F,$
we have 
\begin{equation*}
{\small \tau +F}\left( \int_{0}^{D(z,Tz)}{\small \varphi (t)dt}\right) 
{\small \leq F}\left( \int_{0}^{D(z,Tz)}{\small \varphi (t)dt}\right) 
{\small ,}
\end{equation*}%
which is a contradiction. Thus, we get $z\in Tz.$ That is, $T$ has a fixed
point.
\end{proof}

\begin{example}
Let $X=\left[ 0,1\right] $ and $d(x,y)=\left\vert x-y\right\vert .$ Define $%
T:X\rightarrow K(X)$ multivalued mapping by 
\begin{equation*}
{\small Tx=}\left[ \frac{x}{4}{\small ,}\frac{x+1}{2}\right] {\small .}
\end{equation*}%
For $x_{0}=0$ and $x_{1}=1$ arbitrary points, we have $T_{0}=\left[ 0,\frac{1%
}{2}\right] $ and $T_{1}=\left[ \frac{1}{4},1\right] .$%
\begin{eqnarray*}
{\small H(T}_{0}{\small ,T}_{1}) &{\small =}&\max \left\{ \underset{x\in
T_{0}}{\sup }{\small D}\left( x,T_{1}\right) {\small ,}\underset{y\in T_{1}}{%
\sup }{\small D}\left( y,T_{0}\right) \right\} \\
&{\small =}&\max \left\{ \underset{x\in T_{0}}{\sup }\underset{y\in T_{1}}{%
\inf }d(x,y),\underset{y\in T_{1}}{\sup }\underset{x\in T_{0}}{\inf }%
d(x,y)\right\} \\
&=&\frac{1}{4}.
\end{eqnarray*}%
Now, we show that mapping $T$ under this condition is not generalized
multivalued integral type mapping. Afterward, we will show that it is a
contraction together with $F$. By using Theorem \ref{ojha} and $\varphi
(t)=1 $ for all $t\in 
\mathbb{R}
$, we have 
\begin{equation*}
\int_{0}^{H(T_{0},T_{1})}{\small dt}{\small \leq \alpha }%
\int_{0}^{d(x_{0},x_{1})}{\small dt}\text{ {\small and} }{\small \alpha \geq 
}\frac{1}{4}{\small .}
\end{equation*}%
It can be shown that multivalued integral type mapping above holds for $%
\alpha \in \left[ \frac{1}{4},1\right) .$ In Theorem \ref{ojha}, it holds $%
0\leq \alpha <1.$ This is a contraction. Under same condition; we will apply 
$F$-contraction. Let us take $F(\alpha )=\ln \alpha $ and $\tau \in \left(
0,1.39\right) $ also $\varphi (t)=1$ for all $t\in 
\mathbb{R}
,$ we have 
\begin{equation*}
{\small F}\left( \int_{0}^{H(T_{0},T_{1})}{\small dt}\right) {\small \leq F}%
\left( \int_{0}^{d(x_{0},x_{1})}{\small dt}\right) {\small -\tau .}
\end{equation*}%
This inequality holds three condition of $F$-contraction:
\end{example}

\begin{itemize}
\item[$i)$] It is easy to see $(F1)$.

\item[$ii)$] Let $x_{n}=\left( \frac{1}{n}\right) \in \left[ 0,1\right] ,$ $%
n=1,2,\ldots $ and $x_{0}=0$. By taking $x_{n}=\left( \frac{1}{n}\right) ,$
we get 
\begin{eqnarray*}
{\small Tx}_{n} &{\small =}&\left[ \frac{1}{4n}{\small ,}\frac{n+1}{2n}%
\right] \\
{\small H(Tx}_{n}{\small ,Tx}_{n+1}{\small )} &{\small =}&\frac{1}{4n(n+1)}%
{\small .}
\end{eqnarray*}%
If we denote that $\gamma _{n}=\int_{0}^{\frac{1}{4n(n+1)}}\varphi (t)dt,$%
{\small \ }$\varphi (t)=1$ for all $t\in 
\mathbb{R}
,$ then{\small \ }$\gamma _{n}=\left( \frac{1}{4n(n+1)}\right) .$ By using
some calculations, we conclude that following inequality:%
\begin{equation}
{\small F}\left( \frac{1}{4n(n+1)}\right) {\small \leq F}\left( \frac{1}{%
4(n-1)n}\right) {\small -\tau \leq \cdots \leq F(1)-n\tau .}  \label{on}
\end{equation}%
By taking{\small \ }$F(\alpha )=\ln \alpha ,$ we have%
\begin{equation*}
{\small 0\leq F}\left( \frac{1}{4n(n+1)}\right) {\small \leq }\ln {\small %
1-n\tau .}
\end{equation*}%
From (\ref{on}), letting $n\rightarrow \infty $, we obtain that%
\begin{equation*}
\underset{n\rightarrow \infty }{\lim }\ln \left( \frac{1}{4n(n+1)}\right) 
{\small =-\infty .}
\end{equation*}%
Also for $\varphi (t)=1,${\small \ }$\gamma _{n}=\left( \frac{1}{4n(n+1)}%
\right) $ and $\underset{n\rightarrow \infty }{\lim }\frac{1}{4n(n+1)}%
{\small =0.}$

\item[$iii)$] We will show that {\small \bigskip }$\underset{n\rightarrow
\infty }{\lim }${\small \ }$\gamma _{n}^{k}F(\gamma _{n})=0.$ Letting $%
{\small n\rightarrow \infty }$, we have 
\begin{equation*}
\underset{n\rightarrow \infty }{\lim }{\small \gamma }_{n}^{k}{\small %
F(\gamma }_{n}{\small )=}\underset{n\rightarrow \infty }{\lim }\frac{1}{%
\left[ 4{\small n(n+1)}\right] ^{k}}\ln \left( \frac{1}{4{\small n(n+1})}%
\right) {\small =0.}
\end{equation*}
\end{itemize}

Since{\small \ }$x_{n}=\left( \frac{1}{n}\right) \in \left[ 0,1\right] $ is
a Cauchy sequence and $\left( X,d\right) $ is complete metric space, taking
the limit as ${\small n\rightarrow \infty }$ we have $\frac{1}{n}$ $%
\rightarrow 0$ in $\left[ 0,1\right] $. However, since ${\small Tx}_{n}%
{\small =}\left[ \frac{1}{4n}{\small ,}\frac{n+1}{2n}\right] $ and ${\small %
Tx}_{n}{\small \in K(X)},$ then $\underset{n\rightarrow \infty }{\lim }%
{\small Tx}_{n}{\small =Tx,}$ that is, we have 
\begin{equation*}
\underset{n\rightarrow \infty }{\lim }\left[ \frac{1}{4n}{\small ,}\frac{n+1%
}{2n}\right] {\small =}\left[ 0,\frac{1}{2}\right] {\small .}
\end{equation*}%
As a result we get{\small \ }$0\in \left[ 0,\frac{1}{2}\right] $, that is, $%
T $ has a fixed point under $F$-contraction.

\end{document}